\documentclass[12pt]{article}
\usepackage{amssymb,amsmath}
\usepackage{hyperref,url}
\input{amssym.def}\input{amssym}
\usepackage{graphics}
\newif\ifcover\covertrue\coverfalse

\newif\ifpic\pictrue
\ifpic
\usepackage{graphicx, tikz}
\fi

\def\subjclass#1{\par\medskip
\noindent\textbf{Mathematics Subject Classification (2010):} #1}
\def\keywords#1{\par\medskip
\noindent\textbf{Keywords.} #1}

\newcommand{\D}{{\mathbb D}}

\newcommand{\N}{{\mathbb N}}

\newcommand{\R}{{\mathbb R}}
\newcommand{\T}{{\mathbb T}}
\newcommand{\Z}{{\mathbb Z}}
\renewcommand{\S}{{\mathbb S}}

\newcommand\cB{{\mathcal B}}
\newcommand\cC{{\mathcal C}}

\newcommand\cE{{\mathcal E}}
\newcommand\cF{{\mathcal F}}
\newcommand\cG{{\mathcal G}}

\newcommand\cK{{\mathcal K}}

\newcommand\cU{{\mathcal U}}

\newcommand\bP{{\mathbb P}}

\newcommand\bS{{\mathbb S}}

\newcommand{\Crit}{\operatorname{Crit}}

\newcommand{\diam}{\operatorname{diam}}

\newtheorem{theorem}{Theorem}[section]
\newtheorem{prop}{Proposition}[section]
\newtheorem{corollary}{Corollary}[section]
\newtheorem{example}{Example}[section]
\newtheorem{lemma}{Lemma}[section]
\newtheorem{remark}{Remark}[section]
\newtheorem{definition}{Definition}[section]
\newenvironment{proof}{\noindent {\bf Proof.}}{ \hfill $\Box$\\ }

\newcommand{\orb}{\operatorname{orb}}
\def\aa{\omega}
\def\eps{\varepsilon}

\def\hloc{h_{\text{\tiny loc}}}
\def\htop{h_{\text{\tiny top}}}

\def\tl{translocal }
\def\Tl{Translocal }

\def\red{\color{black}}
\def\blue{\color{black}}

\title{On some aspects of local thermodynamical formalism}
\author{Andrzej Bi\'s
\thanks{Faculty of Mathematics and Computer Science, University of \L\'od\'z,
 Banacha 22, 90-238 \L\'od\'z, Poland;
 {\it andrzej.bis@wmii.uni.lodz.pl}}
\and
Henk Bruin
\thanks{Faculty of Mathematics, University of Vienna,
  		Oskar Morgensternplatz 1, 1090 Vienna, Austria; {\it henk.bruin@univie.ac.at}}
		}

\date{\today}

\begin{document}

\maketitle

\begin{abstract}
 In 2007, Ye \& Zhang introduced a version of local topological entropy. Since their entropy function is, as we show under mild conditions, constant for topologically transitive dynamical systems, we propose to adjust the notion in a way that does not neglect the initial transient part of an orbit.
 We investigate the properties of this ``transient'' version, which we call \tl entropy, and compute it in terms of Lyapunov exponents for various dynamical systems.
 We also investigate how this adjustment affects measure-theoretic local (Brin-Katok) entropy and local pressure functions, generalizing some partial variational principles of Ma \& Wen.
\end{abstract}


\section{Introduction and motivation}\label{sec:intro}

Topological and measure-theoretic entropy can be interpreted as indications of complexity of
a map $f:X \to X$, but they provide a global view, without specifying on which subsets of $X$ the map is ``more complicated'' or ``less complicated''.
This motivated the study of local versions of entropy and of pressure, as we will discuss presently.

Topological entropy is related to measure-theoretic entropy via the variational principle:
If $f$ is a continuous map of a compact metric space $(X,d)$ (an assumption we will make throughout the paper), then $\htop(f) = \sup_\mu h_\mu(f)$, where the supremum is taken
over all $f$-invariant probability measures $\mu$ on $X$.
Measures $\mu$ such that $\htop(f) = h_{\mu}(f)$ are called {\em measures of maximal entropy} or {\em maximal measures}.
In thermodynamic formalism, one tries to maximize  the sum of entropy and energy (represented by an integral over a potential function $\phi:X \to \R$), called {\em free energy}
or {\em measure-theoretic pressure}.
The {\em topological pressure} can then be defined by the variational principle as
$P_{top}(f) = \sup_\mu \{ h_\mu + \int_X \phi \, d\mu\}$, see e.g.\ \cite{Rue04,walters73}.

Local versions of topological entropy were used by
Bowen \cite{Bow} to get an upper estimate of the difference
between $h_{\mu}(f)$ and the entropy of a partition  with diameter less than $\eps$ and by Misiurewicz \cite{Mis}, who proved that vanishing of his local entropy implies existence of {\blue a} maximal measure.
Different versions were used by Newhouse \cite{N89}, and by Buzzi \& Ruette \cite{Buz97, BR}, to address questions about the (unique) existence of a measure of maximal entropy. All these notions rely on {\em Bowen balls} (also called
{\em dynamical balls}):
{\red
$$
B_n(z;\eps) = \{ z' \in X : d(f^j(z), f^j(z') < \eps \text{ for all } 0 \leq j < n\}.
$$
}
A local view on measure-theoretic entropy goes back to at least Brin \& Katok \cite{BK83}.
In particular, Brin-Katok {\blue local} entropy $h_\mu(f,x)$ (see~\eqref{eq:BK} in Section~\ref{sec:BK} for the definition), when integrated over the space yields the global measure-theoretic entropy:
$$
\int h_\mu(f,x) \, d\mu(x) = h_\mu(f).
$$
Variational principles have been derived for this local version as well, see e.g.\ \cite{FH12}.

\subsection{The Ye \& Zhang entropy function and \tl entropy.}
In 2007, Ye \& Zhang \cite{YZ07} introduced a new version of local entropy,
which they called {\em entropy function} $\htop:X \to [0,\infty]$.

Their definition, see~\eqref{eq:htop} below, is the starting point of our paper, but one soon discovers that under mild conditions the entropy function is constant: {\blue for any $z\in X$ one has} $\htop(z) \equiv \htop(f)$ for very natural dynamical systems, see  {\blue Corollary~\ref{cor:hconstant}}.
 It is not straightforward to find maps with non-constant entropy functions, see Section~\ref{sec:hconst}.
The reason is that in the Ye \& Zhang approach, only the limit behaviour of
the orbit of $z$ counts, not the ``transient'' behaviour at finite time scales.
The entropy function does not see if entropy is created ``immediately'' in small neighbourhoods of $z$, or only at very distant time scales.

For this reason, we propose an adjustment to Ye \& Zhang's definition,
which we call the {\em \tl entropy function} and denote it as $h_{\aa}$,
that takes into account the time scale needed for Bowen balls to reach unit size. As a result, $h_{\aa}(z)$ depends non-trivially on $z \in X$, and Lyapunov exponents at $z$
play a central role. This interaction is most transparent in smooth one-dimensional maps,
see Theorem~\ref{thm:circle},
where each point has at most one Lyapunov exponent (because there is only one direction in which derivatives can be taken).
As representatives of higher-dimensional maps, we study toral automorphisms $F_A: \T^d \to \T^d$ based on a $d \times d$ matrix $A$.
Theorem~\ref{thm:toral} shows that its \tl entropy function is determined by the Lyapunov exponents, i.e., the logarithms of the eigenvalues of $A$.

In the second half of the paper, we turn to the local approach to pressure and study the \tl version of pressure.
 Ma \& Wen in \cite{MW} obtained a partial variational principle for topological entropy. They noticed
(see Theorem~\ref{thm:MW}) the relations between topological entropy of a map and its upper and lower measure-theoretic entropies.
Using a Carath\'eodory-like construction, elaborated by Pesin \cite{Pes}, we define a dimensional type of topological pressure $P_Z(f, \phi)$ for a continuous map $f:X \to X$ with potential $\phi$, restricted to subsets $Z \subset X$. Fixing a Borel probability measure $\mu$ on $X$ we generalize the Ma \& Wen result by means of upper and lower local measure-theoretic pressures, see Theorem~\ref{thm:aaMW}.
In Theorem~\ref{thm:tlMW} we show the same result for \tl version of local
measure-theoretic pressures.

\subsection{Structure of the paper}
Our paper is organized as follows. In Section~\ref{sec:localentro}, we introduce the entropy function by Ye \& Zhang, and study when it is constant and when not.
In Section~\ref{sec:aalocal}, we give our adaptation of the entropy function, and give its basic and more advanced properties in the remaining subsections of Section~\ref{sec:localentro}.

Section~\ref{sec:localpressure} is devoted to local measure-theoretic approach to pressure.
Section~\ref{sec:BK} deals with measure-theoretic versions of local entropy, starting with the Brin-Katok entropy. We discuss (local) topological pressure in Section~\ref{sec:toppressure} and give a variational principle in Section~\ref{sec:locpressure}.
The variational principle of our \tl adaptation is covered in Section~\ref{sec:aapressure}.
We finish the paper in Section~\ref{sec:discussion} with a discussion and comparison of various notions of local entropy in the literature.

\section{Local topological entropy in sense of Ye \& Zhang}\label{sec:localentro}

Let $f:X \to X$ be a continuous map on a compact metric space $(X,d)$.
A set $A$ is called {\em $(n,\eps)$-separated} if $d(f^j(x),f^j(y)) > \eps$ for some $0 \leq j < n$ and all distinct $x, y \in A$.
For a closed subset $K \subset X$, we define the topological entropy\footnote{This is in line with the Bowen-Dinaburg definition \cite{Bow71, Din70}, used in e.g.\ \cite{PP84,TV99,TV03}, who were eventually interested in the multifractal properties of a measure-theoretic version of local entropy, in the line of the Brin-Katok entropy which we discuss in Section~\ref{sec:BK}.}
$\htop(f,K)$
restricted to $K$ as the limit of exponential growth rates of $(n,\eps)$-separated subsets of $K$ as $\eps \to 0$, that is
\begin{equation}\label{eq:Bowen}
\htop(f,K) = \lim_{\eps \to  0}\limsup_{n\to\infty} \frac1n \log S(n,\eps,K),
\end{equation}
where $S(n,\eps,K)$ is the maximal cardinality of $(n,\eps)$-separated subsets of $K$.

{\blue  Can one modify~\eqref{eq:Bowen}
  into a pointwise definition of topological entropy by shrinking $K$  to a point?} One way of doing this is
\begin{equation}\label{eq:notle}
\lim_{\delta \to 0} \inf_{K \in \cK_\delta(x)} \htop(f,K),
\end{equation}
where {\blue $\cK_\delta(x)$} is the collection of closed neighborhoods of $x$ of diameter $\leq \delta$.
A point $x \in X$ is an {\em entropy point} if there are arbitrarily small closed neighbourhoods $K$ such that $\htop(f,K) > 0$ and a {\em full entropy point} if in addition $\htop(f,K) = \htop(f)$, see  \cite[Section 2]{YZ07}. Hence, only maps with positive entropy can have full entropy points.

{\blue A map $f:X \to X$ is {\em topologically exact} or {\em locally eventually onto} if for every open $U \subset X$, there is $N \in \N$ such that $f^N(U) = X$.}
However, 
if $f$ is locally eventually onto, then there is $N_K$ such that $f^{N_K}(K) = X$.
Notice that $(n-N_k,\eps)$-separated points $y_1,y_2 \in f^{N_K}(K)$ determine points $x_1,x_2 \in K$, with $y_1=f^{N_K}(x_1)$ and $y_2=f^{N_K}(x_2)$, such that $x_1, x_2$ are $(n,\eps)$-separated. Therefore {\red $S(n,\eps,K)\ge S(n-N_k,\eps, f^{N_K}(K))$.}
Thus the quantity in~\eqref{eq:notle} is constant and equal to $\htop(f)$.

Instead, in~\cite[Definition 4.1]{YZ07} the {\em local entropy function} is defined as {\blue
\begin{equation}\label{eq:htop}
\htop(x) = \lim_{\eps \to 0} \htop(x,\eps), \qquad \text{for~~~}
\htop(x,\eps) = \lim_{\delta \to 0} \inf_{K_x \in \cK_{\delta}(x)} \limsup_{n \to \infty} \frac1n \log S(n,\eps,K_x).
\end{equation}
}
This differs from~\eqref{eq:notle} in the order in which the limits $\eps \to 0$ and $\delta \to 0$ are taken.
In general, $\htop(x) = \htop(f(x))$, {\blue for any point $x \in X$}, so the local entropy is constant along forward orbits, see \cite[Proposition 4.4(1)]{YZ07}.
In fact, we have {\blue the following. }

\begin{lemma}
 The entropy function takes its minimum at transitive points, i.e., points with a dense orbit.
\end{lemma}
{\blue
\begin{proof}
Fix $\delta >0.$ Choose a point $x \in X$ with a dense orbit and
 a point $y \in X$ with  $K_y \in \cK_{\delta}(y).$ Then, there exists $n_0\geq 0$  such that
 $f^{n_0}(x) \in K_y$ and  a set  $K_x \in \cK_{\delta}(x)$
 such that  $f^{n_0}(K_x) \subset K_y$. For any $(n,\eps)$-separated points $y_1,y_2 \in f^{n_0}(K_x)$ there exist $(n+n_0,\eps)$-separated points $x_1,x_2 \in K_x$
 (i.e. $y_1= f^{n_0}(x_1)$ and $y_2 = f^{n_0}(x_2)$). Thus,
 $$S(n+n_0,\eps,K_x)\le S(n,\eps,f^{n_0}(K_x)) \le S(n,\eps,K_y)$$
 and taking limsup with $n \to \infty$ we get
 $$\limsup_{n \to \infty} \frac{n +n_0}{n} \cdot \frac{1}{n+n_0} \log S(n+n_0,\eps,K_x) \le \limsup_{n \to \infty}  \frac{1}{n} \log S(n,\eps,K_y).$$
 Applying~\eqref{eq:htop} and taking the infimum over $K \in \cK_\delta$ on both sides, we obtain
 $$
 \htop(x,\eps)  =\inf_{K_x \in \cK_{\delta}(x)}\limsup_{n \to \infty}  \frac{1}{n} \log S(n+n_0,\eps,K_x) \le \inf_{K_y \in \cK_{\delta}(y)}\limsup_{n \to \infty}  \frac{1}{n} \log S(n,\eps,K_y)
 =\htop(y,\eps),
 $$
 and hence $\htop(x) \leq \htop(y)$.
\end{proof}
 }

\subsection{When is the entropy function non-constant?}\label{sec:hconst}

One can argue the merits of local entropy function that is equal to $\htop(f)$ everywhere.
In this case, the behavior near a single orbit  suffices to estimate the topological entropy, but on the other hand, such an entropy function cannot indicate any inhomogeneity in where in the system, and how rapidly, chaos is created.

The next theorem shows that the Ye \& Zhang local entropy function is indeed constant
under mild assumptions.
A version of this theorem, for much more specific dynamical systems, can be found in \cite{BV19}.

Topological exactness gives a simple sufficient condition for the entropy function to be constant.

\begin{lemma}\label{lem:leo}
 If $f:X \to X$ is topologically exact {\blue map, }then $\htop(z) \equiv\htop(f)$ for any $z\in X$.
\end{lemma}

\begin{proof}
 Given a closed neighbourhood $K \owns z$, take $N$ such that $
 f^N(K) \supset f^N(\mathring{K}) = X$, where $\mathring{K}$ indicates the interior {\blue of $K.$ }
 Then $\frac1n \log S(n,\eps,K) \geq \frac1n \log S(n-N, \eps, X) =
 \frac{n-N}{n} \frac{1}{n-N} \log S(n-N,\eps, X)$.
 So after taking limits $n \to \infty$ and $\eps \to 0$, we find $\htop(z) \geq \htop(f)$.
\end{proof}

\medskip
{\red  Ye \& Zhang \cite[Theorem 3.7]{YZ07} prove, based on a lemma by Katok, that
for every ergodic $f$-invariant measure $\mu$:
\begin{equation}\label{YZ3.7}
 \lim_{\eps\to 0} \inf\{ \limsup_{n\to\infty} \frac1n \log S(n,\eps,K) :
 K \text{ is a Borel set with } \mu(K) > 0 \} \geq h_\mu(f).
\end{equation}
This allows us to derive the following corollary:

\begin{corollary}\label{cor:hconstant}
 Let $f:X \to X$ be a continuous topological transitive map defined on a compact metric space $(X,d)$, and assume that it has at least one measure of maximal entropy.
 If $\htop(f) = 0$ or if each measure of maximal entropy is fully supported, then $\htop(x) = \htop(f)$ for every $x \in X$.
\end{corollary}
}

{\blue
\begin{proof}
If $\htop(f) = 0,$ then of course the corollary is true.
Since all measure of maximal entropy are fully supported, certainly there is a fully supported {\bf ergodic}
$f$-invariant measure $\mu$ of maximal entropy.
Therefore $\mu(K) > 0$ for every compact neighbourhood of every $z \in X$. Fix $\delta>0$ and $x\in X.$ Since
\begin{align*}
 \inf & \{ \limsup_{n\to\infty} \frac1n \log S(n,\eps,K_x) :
 K_x  \in \cK_{\delta}(x)~and~ \mu(K_x) > 0 \} \\
 & \ge \inf\{ \limsup_{n\to\infty} \frac1n \log S(n,\eps,K) :
 K \text{ is a Borel set with } \mu(K) > 0 \},
 \end{align*}
we have
 \begin{align*}
 \lim_{\delta \to 0} & \inf\{ \limsup_{n\to\infty} \frac1n \log S(n,\eps,K_x) :
 K_x \in \cK_{\delta}(x)~and~ \mu(K_x) > 0 \} \\
 & \ge  \inf\{ \limsup_{n\to\infty} \frac1n \log S(n,\eps,K) :
 K \text{ is a Borel set with } \mu(K) > 0 \}.
\end{align*}
Taking limit with $\eps \to 0$ we get
\begin{eqnarray*}
\htop(x) &=& \lim_{\eps \to 0}\lim_{\delta \to 0} \inf\{ \limsup_{n\to\infty} \frac1n \log S(n,\eps,K_x) :
 K_x \in \cK_{\delta}(x)~and~ \mu(K_x) > 0 \} \\
 &\ge& \lim_{\eps \to 0}  \inf\{ \limsup_{n\to\infty} \frac1n \log S(n,\eps,K) :  K \text{ is a Borel set with } \mu(K) > 0 \}.
 \end{eqnarray*}
Formula~\eqref{YZ3.7} and the assumption that $\mu$ is a measure of maximal entropy yield
 $$
   \htop(f) \geq \htop(x) \geq h_\mu(f) = \htop(f).
$$
\end{proof}
}

The assumption that the measures of maximal entropy are fully supported is not entirely automatic, although it is conjectured that {\bf on manifolds} the measure of maximal entropy (if existent)
is automatically fully supported if the map is topologically transitive, see \cite{Y11} 
for results supporting this conjecture. In the context of subshifts $(X,\sigma)$ of positive entropy (so $X$ is a Cantor set, not a manifold), Kwietniak et al.~\cite{KOR14}
presented a counter-example with all of its measures of maximal entropy supported on a proper subsets of $X$.
Loosely inspired by their construction, we have the following topologically transitive (and in fact coded)
subshift with a non-constant entropy function.

\begin{example}
 We construct a coded subshift of $\{0,1,2\}^\Z$ on which the local entropy function is not constant.
 Let $\{w_k\}_{k \in \N}$ be an enumeration of all words in $\bigcup_{n \geq 1}\{0,1\}^n$ such that $|w_k| \leq |w_{k+1}|$ for all $k \in \N$.
Let $\cC = \{ C_k \}_{k \in \N}$ be the collection of code words, where
 $$
 C_k =  20^{(10+k)!} w_k 0^{(10+k)!} 2.
 $$
 A {\em coded shift} $X_{\cC}$ is the shift of which the language consists of subwords of free concatenations of code words, see~\cite{LM} and \cite{Pav18}.
 The structure of the code words in $\cC$ implies that $22$ is a {\em synchronizing word}
 (i.e., if $u22$ and $22v$ are both allowed words in $X_{\cC}$, then so is $u22v$), but more importantly, $\htop(X_{\cC},\sigma) \geq \log 2$, because the full shift on $\{0,1\}$ is a subshift of $X_{\cC}$.

 Now consider the cylinder set $[2.2] = \{ z \in \{ 0,1,2 \}^{\Z} : z_{-1}=z_0=2\}$ and let $z \in [2.2]$ be such that $\orb_\sigma(z)$
 is dense in $X_{\cC}$ and $K \subset [2.2]$ is a closed neighbourhood of $z$.
 From the theory of coded shifts, see \cite[Section 3.3]{Bruin23},
 we can compute the exponential growth rate of the number of centered words $x$ of length $2n+1$ and with $x_{-1} = x_0 = 2$.
 Namely, this is the unique positive solution $h$ of $\sum_k e^{-h |C_k|} = 1$, and it is clear that $h < \log 2$.
 It follows that $\htop(z) \leq h < \htop(X_{\cC},\sigma)$.
In this case, work of Pavlov \cite[Theorem 1.1]{Pav18} implies that
its measures of maximal entropy are supported on $\overline{\orb_\sigma(z)}\setminus \orb_\sigma(z)$, i.e., $\{ 0, 1\}^\Z$.
On the other hand, for every closed neighbourhood $K$ of $z$, there is $n$ such that
$\sigma^n(K) \supset [2.2]$, and therefore
{\blue $\htop(z) =  \log h < \htop(X_{\cC},\sigma) = \log 2$.}
\end{example}

Without the assumption of transitivity, maps with non-constant entropy are not hard to find.
A simple example of a map where $\htop(z)$
is not constant, and in fact, $\sup_z \htop(z) = \infty$, is given in Figure~\ref{fig:infent}, see \cite{H20} for H\"older continuous maps of this type.
Maps of this type (i.e., with infinitely many transitive components, each with its own value for local entropy, and infinite entropy altogether), are $C^0$-generic on any manifold of dimension $\geq 2$, see \cite{FHT21}.
For such examples, the values $E$ that local entropy assumes forms a countable subset of $[0,\infty)$.
With some more work
(but giving up on genericity) one can find smooth examples where
$E$ is dense in $[0,\infty)$, as well.

\ifpic
\begin{figure}[h]
\begin{center}
\begin{tikzpicture}[scale=0.5]
\draw[-] (16,0) -- (0,0) -- (0,16) -- (16,16) -- (16,0);
\draw[-,dashed] (0,0) -- (16,16);
\draw[-,dashed] (8,0) -- (8,16); \draw[-,dashed] (0,8) -- (16,8);
\draw[-,dashed] (4,0) -- (4,8); \draw[-,dashed] (0,4) -- (8,4);
\draw[-,dashed] (2,0) -- (2,4); \draw[-,dashed] (0,2) -- (4,2);
\draw[-,dashed] (1,0) -- (1,2); \draw[-,dashed] (0,1) -- (2,1);
\draw[-,dashed] (0.5,0) -- (0.5,1); \draw[-,dashed] (0,0.5) -- (1,0.5);
\draw[-,thick] (16,16) -- (14.13,8) -- (10.66,16) -- (8,8) --
(7.2,4)--(6.4,8)--(5.6,4)--(4.8,8)--(4,4) -- (3.72,2)--(3.44,4) -- (3.16,2)--(2.88,4)--(2.59,2)--(2.29,4)--(2,2)--(1.88,1)--(1.77,2)--(1.66,1)--(1.55,2)--(1.44,1)--(1.33,2)--(1.22,1)--(1.11,2)--(1,1);
\node at (0.5,0.5) {};
 \end{tikzpicture}
 \caption{An infinite entropy map where $\htop(x) = \log(2n+1)$ for $x \in (2^{-n}, 2^{1-n}]$ and $\htop(0)= \infty$.}\label{fig:infent}
\end{center}
\end{figure}
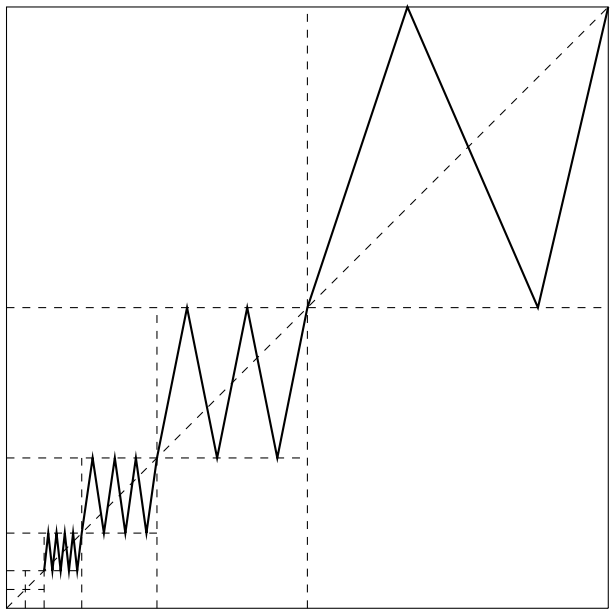
\fi

\subsection{\Tl entropy}\label{sec:aalocal}

\begin{definition}
For $\aa \geq 0$ and $z \in X$, the {\em upper \tl entropy} at $z \in X$ is
{\blue
$$
 \overline h_{\aa}(z) =  \lim_{\eps \to 0} \limsup_{n\to\infty} \frac1n \log S(n,\eps, \overline{B(z, e^{-\aa n})}).
$$
 }
The {\em lower \tl entropy} $\underline h_{\aa}(z)$ at $z \in X$ is the same with $\liminf$ instead of $\limsup$. If the two are the same, then
$h_{\aa}(z) := \overline h_{\aa}(z) = \underline h_{\aa}(z)$
is the {\em \tl entropy} at $z$.
\end{definition}

Clearly $\overline h_{\aa}(z)$ is decreasing in $\aa$.
Compared to the quantity~\eqref{eq:notle}, $\htop(x)$ is obtained by swapping the limit $\lim_{\eps\to 0}$ and $\inf_K$.

As a result {\blue we get the following. }

\begin{lemma}\label{lem:Fatou}
Let $f:X \to X$ be a continuous map on a compact metric space.
For all $x \in X$ and $\aa \geq 0$, we have  {\blue
\begin{eqnarray}\label{eq:Fatou}
0 &=&  \lim_{\eps \to 0} \liminf_{n\to\infty} \inf \{   \frac1n \log S(n,\eps,K_x): \delta>0,~ K_x \in \cK_{\delta}(x)\}  \nonumber  \\
&\leq&
\underline h_{\aa}(x)
\ \leq \ \overline h_{\aa}(x) \\[1mm]
&\leq&
\inf\{h_{top}(f,K):  \delta>0,~ K_x \in \cK_{\delta}(x)\}
\leq \ \htop(x).
\nonumber
\end{eqnarray}
}
\end{lemma}

\begin{proof}
If $K$ is sufficiently small compared to $n$, then
$S(n,\eps,K) = 1$, so the first equality in~\eqref{eq:Fatou} is obvious.
 Since $S(n,\eps,K) \geq 1$, we have $\frac1n \log S(n,\eps,K) \geq 0$ for all $\eps > 0$, $n \geq 1$ and closed neighbourhoods $K$ of $x$. Also,
 $\frac1n \log S(n,\eps,K)$ is  decreasing in $n$ and in $\eps$, but increasing in $K$. {\blue
 Now we prove the second inequality  in~\eqref{eq:Fatou}. To this aim we fix $n \in \mathbb{N}$ and $\eps >0,$ then there exists $\delta_n < 2e^{-\aa n}$ such that for any $ K_x \in \cK_{\delta_n}(x)$
 the inclusion $K_x \subset B(x, e^{-\aa n})$ holds. Therefore, for any $n \in \N$ we get
 $$\frac{1}{n} \log S(n,\eps, K_x) \le \frac{1}{n} \log S(n,\eps, B(x, e^{-\aa n})).$$
 Taking infimum over $\delta_n>0$ and $K_x \in \cK_{\delta_n}(x)$ we obtain
 $$
 \inf \left\{\frac{1}{n} \log S(n,\eps, K_x) : \delta_n>0 \text{ and } K_x \in \cK_{\delta_n}(x) \right\} \le \frac{1}{n} \log S(n,\eps, B(x, e^{-\aa n})).
 $$
 Finally, taking liminf $n \to \infty$ and limit  $\eps \to 0$ we get the second inequality in the statement of the lemma. The  third inequality is obvious.
 
 It remains to show the last inequality in~\eqref{eq:Fatou}. Fix $\delta>0$ and choose $K \in \cK_{\delta}(x),$ then there exists $n_{K,\delta} \in \mathbb{N}$ such that for
 any for any $n \ge n_{K,\delta}$ the inequality $S(n,\eps, B(x, e^{-\aa n}) \le S(n,\eps,K)$ holds. Therefore, for any $n \ge n_{K,\delta}$
 $$
 \frac{1}{n} \log S(n,\eps, B(x, e^{-\aa n}) \le \frac{1}{n} \log S(n,\eps, K).
 $$
 Thus, taking infimum over  $K \in \cK_{\delta}(x),$ for large enough $n$ we get
 $$
  \frac{1}{n} \log S(n,\eps, B(x, e^{-\aa n}) \le \inf \{\frac{1}{n} \log S(n,\eps, K): ~K \in \cK_{\delta}(x)\},
$$
so
$$
\limsup_{n \to \infty}\frac{1}{n} \log S(n,\eps, B(x, e^{-\aa n}) \le \inf \{\limsup_{n \to \infty}\frac{1}{n} \log S(n,\eps, K): ~K \in \cK_{\delta}(x)\}
$$
and
$$
\limsup_{n \to \infty}\frac{1}{n} \log S(n,\eps, B(x, e^{-\aa n}) \le \lim_{\delta \to 0}\inf \{\limsup_{n \to \infty}\frac{1}{n} \log S(n,\eps, K): ~K \in \cK_{\delta}(x)\}.
$$
 Finally, taking limit $\eps\to 0$ we get the last inequality in~\eqref{eq:Fatou}.  }
\end{proof}

The next example explores the role of Lyapunov exponents in the \tl entropy function.
For one-dimensional maps, the
Lyapunov exponent is defined as
$$
\lambda(x) = \lim_{n\to\infty}\frac1n \log| Df^n(x)|,
$$
provided the derivatives and the limit exist.

\begin{example}\label{exam:log3} (i)
 Let $f:\S^1 \to\S^1$, $x \mapsto 3x \mod 1$.
 For any $x \in \S^1$ and $K = [x-\delta,x+\delta]$,
 we have $f^{n_0}(K) = \S^1$ for $n_0 \geq -\frac{\log 2\delta}{\log 3}$, and $S(n,\eps,K) \geq 3^{n-n_0}/\eps$.
 Therefore
 $\htop(f,K) = \log 3 = \lambda(x)$ for every $x \in \S^1$, and so every point is a full entropy point.
 Regarding $h_{\aa}(z)$, if $K_n(z) = \overline{B(z,e^{-\aa n}})$, then
 after $k = \aa n/\log 3$ iterates, $f^k(K_n(z)) = \S^1$.
 Therefore
 \begin{eqnarray*}
 \frac1n \log S(n,\eps, K_n(z)) &\sim&
 \frac1n \log S(n-k,\eps, [0,1]) \\
 &=& (1-\frac{\aa}{\log 3}) \frac{1}{n-k}\log S(n-k,\eps, [0,1]) \to  (1-\frac{\aa}{\log 3}) \htop(f).
 \end{eqnarray*}
This is independent of $z$ because the Lyapunov exponent is the same for every $z$.

(ii) However, if $g:\S^1 \to \S^1$ is given by
 $$
 g(x) = \begin{cases}
         2x & \text{ if }\ 0 \leq x < \frac12,\\[1mm]
           4x-2 & \text{ if }\ \frac12 \leq x < \frac34,\\[1mm]
           4x-3 & \text{ if }\ \frac34 \leq x < 1,
        \end{cases}
$$
then $\htop(g,K) = \log 3$, independent of the Lyapunov exponent which is non-constant for this map. Again every point is a full entropy point.
But this time, the Lyapunov exponent $\lambda(z)$ varies with the point, and it takes $k \sim \frac{\aa}{\lambda(z)}$ iterates for
$K_n(z)$ to reach large scale. If the limit $\lambda(z)$ exists,
then $h_{\aa}(z) =  (1-\frac{\aa}{\lambda(z)}) \htop(f)$.

(iii) The same method shows that $h_{\aa}(z) = 0$ for $z = 0$ and the Pomeau-Manneville map
$$
g_{PM}(x) = \begin{cases}
         \frac{x}{1-x} & \text{ if }\ 0 \leq x \leq \frac12,\\[1mm]
           2x-1 & \text{ if }\ \frac12 < x < 1.
        \end{cases}
$$
For this map, 
 {\blue
$g_{PM}(\frac1n) = \frac{1}{n-1}$ for $n \geq 2$,  } so if $K_n(0) = [0, e^{-\aa n}]$,
then it takes $e^{\aa n} \gg n$ iterates to reach unit size.
The above computation  gives $h_{\aa}(0) = 0$ for all $\aa > 0$.

(iv) On the other hand, $h_{\aa}(z) = \log 2$ for $z=0$ and every $\aa \geq 0$ for the map
$$
g_{\sqrt{}}(x) = \begin{cases}
         \sqrt{2x} & \text{ if }\ 0 \leq x < \frac12,\\[1mm]
           2x-1 & \text{ if }\ \frac12 \leq x < 1.
        \end{cases}
$$
This time, {\blue$g^k_{\sqrt{}}(e^{-\aa n}) = 2^{ 1-2^{-k} } e^{-\aa n2^{-k}}$, }
so it takes no more than $k \approx \log n$ iterates to reach large scale.
Hence, $h_{\aa}(0) = \htop(g_{\sqrt{}}) = \log 2$ for all $\aa \geq 0$.

\ifpic
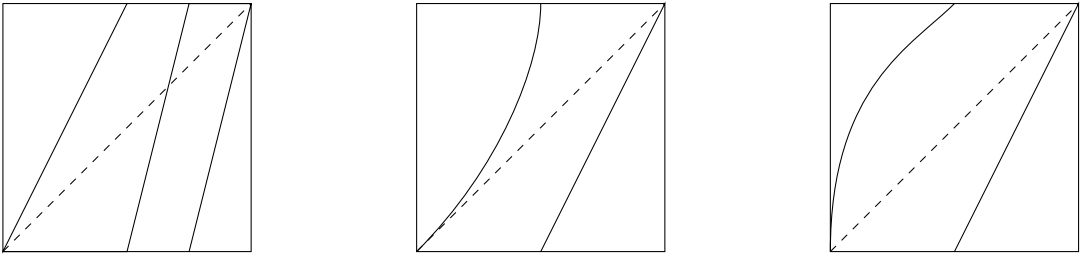
\begin{figure}[ht]
\begin{center}
\begin{tikzpicture}[scale=0.55]
\draw[-] (3,6) -- (0,0) -- (0,6) -- (6,6) -- (6,0)--(0,0) -- (3,0)--(4.5,6)--(6,6) -- (4.5,0);
\draw[-, dashed] (0,0)--(6,6);
\draw[-] (13,0)--(16,6)--(16,0)--(10,0)--(10,6)--(16,6);
\draw[-, dashed] (10,0)--(16,6);
\draw[-] (10,0) .. controls (12,2) and (13, 4.5) .. (13,6);
\draw[-] (23,0)--(26,6)--(26,0)--(20,0)--(20,6)--(26,6);
\draw[-, dashed] (20,0)--(26,6);
\draw[-] (20,0) .. controls (20, 4) and (22,5) .. (23,6);
\end{tikzpicture}
\caption{The maps $g$, $g_{PM}$ and $g_{\sqrt{}}$.}
\label{fig:maps}
\end{center}
\end{figure}
\fi

\end{example}

\subsection{Properties of \tl entropy}\label{sec:aadimension}

As the \tl entropy $h_{\aa}(z) = 0$ if $f$ is constant on a neighbourhood of $z$, we make the assumption:
\begin{quote}
{\bf In the remainder of this paper, $f:X \to X$ is a local homeomorphism.}
\end{quote}

\begin{lemma}
The \tl entropy for iterates of the map $f$ satisfies {\red $\underline{h}_{\aa}(f^R,z) = R \underline{h}_{\aa}(f,z)$ for integers $R \geq 0$ and also for
$\underline{h}_{\aa}(f^{-R},z) = R \underline{h}_{\aa}(f,z)$ for integers $R \geq 0$ if $f$ is invertible. The analogous statements hold for $\overline{h}_{\aa}$.}
\end{lemma}

\begin{proof}
The proof goes as for the analogous statement on topological entropy.
\end{proof}

\begin{definition}
 We say that $u$ has a {\em super-exponential approach} to $v$ if
 $$
 \limsup_{k\to\infty}\ -\frac1k \log d(f^k(u), v) = \infty.
 $$
\end{definition}

Clearly a periodic point $u$ has a {\em super-exponential approach} to 
every $v \in \orb_f(u)$,
but the definition refers to other situations too.
The following example shows this, but also that {\em super-exponential approach}
is not a symmetric or transitive notion.

\begin{example}\label{exam:uvw}
Let $X = \{ 0,1\}^\Z$ and $\sigma:X \to X$ be the left shift.
Let $w = 0^\infty . 0^\infty \in X$ the sequence that is constant $0$.
Let
$$
v = 1^\infty.10 1^{3!-2!} 0^{4!-3!} 1^{5!-4!} 0^{6!-5!}
 1^{7!-6!}  0^{8!-7!} \dots
$$
Note that the number of iterates $2+\sum_{j=1}^{n-1} j!-(j-1)! = (n-1!) + 1$ before the $n$-th block of $0$s is reach is negligible compared to the length $n!-(n-1)! = (n-1) (n-1)!$ of that block,
so $v$ exponentially approaches $w$, but as $w$ is a fixed point,
$w$ does not approach $v$.
Finally, let
$$
u = 1^\infty.1^{2!} v_1\dots v_{3!-2!} 1^{4!-3!}
 v_1\dots v_{5!-4!} 1^{6!-5!} v_1\dots v_{7!-6!} 1^{8!-7!} \dots
$$
Then $u$ has super-exponential approach to $v$, but not to $w$.
\end{example}

\begin{prop}\label{prop:super-approach}
Let $f:X \to X$ be a Lipschitz continuous map on a metric space.
 If $u$ has a {\em super-exponential approach} to $v$, then $\underline{h}_{\aa}(u) \leq \overline{h}_{\aa}(v)$.
\end{prop}

\begin{proof}
If $\aa = 0$, then the \tl entropy function coincides with the topological entropy,
and there is nothing to prove. So assume that $\aa > 0$.

Let $L$ be a global Lipschitz constant of $f$; for simplicity, we assume that $\log 2L > 1+2\aa$.
Choose $\eps > 0 $ and $R > \log L$ arbitrary and pick $k \in \N$ such that $d(f^k(u), v) \leq e^{-Rk} < \eps/4$.
Then for $n = \lceil Rk/\aa \rceil$ we have
$$
f^k (\overline{B(u,e^{-\aa (n+k) })} ) \subset B(f^k(u),e^{-\aa(n+k)+k \log L}) \subset
B(v, 2e^{ -\aa(n+k)+k \log L }) \subset \overline{B(v, e^{-\aa(n-k')})}
$$
for $k' := k(\frac{\log 2L}{\aa} - 1) > k$.

We can assume that {\blue $\diam f^j(\overline{ B(u,e^{-\aa (n+k) }) } ) < \eps$ }for $j \leq k$, so
in the first $k$ iterates, no points in {\blue$\overline{ B(u,e^{-\aa (n+k) }) }$ }
are $\eps$-separated.
Therefore
\begin{eqnarray*}
\frac{1}{n+k} \log S(n+k,\eps, \overline{B(u, e^{-\aa (n+k)})} )
&\leq&
\frac{1}{n+k} \log S(n,\eps, \overline{B(f^k(u), e^{-\aa(n + k) + k\log L} ) }) \\
&\leq& \frac{1}{n+k}
\log \left( S(n,\eps, \overline{B(v,e^{-\aa (n - k')})} ) \right) \\
&\leq& \frac{1}{n+k}
\log \left( S(n-k',\eps, \overline{B(v,e^{-\aa (n - k')})} ) \cdot e^{ (k'+k) \htop(f) } \right) \\
&\leq& \frac{1}{n-k'} \log S(n-k',\eps, \overline{B(v,e^{-\aa(n - k')})} ) + \frac{\htop(f)\log 2L}{R},
\end{eqnarray*}
where  in the last line we used $\frac{k'+k}{n+k} < \frac{\log 2L}{R}$.
Since $\htop(f) \leq \log L$ for Lipschitz maps and $\eps$ and $R$ are arbitrary,
$\underline h_{\aa}(u) \leq \overline h_{\aa}(v)$, as claimed.
\end{proof}

\begin{example}
 This example shows that no other inequality holds in Proposition~\ref{prop:super-approach}.
 Let $v$ and $w$ be the one-sided versions of the sequences $v$ and $w$ in
 Example~\ref{exam:uvw}, and interpret them as the itineraries of two points $p_v \in [0,1]$ and $p_w = 0$ for the map $g$ in Example~\ref{exam:log3}(ii),
 with respect to the partition $I_0 = [0,\frac12)$, $I_1 = [\frac12,\frac34)$ and $I_2 = [\frac34,1]$.
 Then $v$ approaches $w$ super-exponentially, but
 $\overline h_\aa(p_u) = (1-\frac{\aa}{\log 4}) \log 3 > (1-\frac{\aa}{\log 2}) \log 3
 = h_{\aa}(p_w)$ for every $\aa \in (0, \log 2]$.
\end{example}

\subsection{\Tl entropy for one-dimensional maps}\label{sec:1D}

The next theorem is exemplary for the \tl entropy of one-dimensional maps.
It holds by and large also for smooth interval maps with critical points, but dealing with the technicalities of distortion control is not the purpose of this paper, so we only state it for expanding circle maps.

\begin{theorem}\label{thm:circle}
 Let $f:\bS^1 \to \bS^1$ be a $C^2$ expanding circle map
 and $\aa \geq 0$, then for every $z \in \bS^1$
 $$
 \overline h_{\aa}(z) = \max\{\left( 1 - \aa/\overline \lambda(z)\right)  \htop(f)\ , \ 0 \},
 $$
 where $\overline \lambda(z) = \limsup_{n \to \infty} \frac1n \log |Df^n(z)|$ is the upper Lyapunov exponent of $z$,
 and
 $$
 \underline h_{\aa}(z) =\left( 1 - \aa/\underline \lambda(z)\right) \htop(f) \},
 $$
 where $\underline \lambda(z)  = \liminf_{n \to \infty} \frac1n \log |Df^n(z)|$ is the lower Lyapunov exponent of $z$.
\end{theorem}

\begin{proof}
If $\aa = 0$, then $B(z,e^{-\aa n})$ is independent of $n$,
and the topological transitivity gives $\underline h_0(z) = \overline h_0(z) = \htop(f)$ for all $z \in \bS^1$.
So let us fix $\aa > 0$.

Since $f$ is a $C^2$ expanding circle map, there are {\blue$c, c^{+}, c_2 > 0$ }such that
 for all $x \in \S^1$ we have expansion: $c^+ \geq |f'(x)| \geq c > 1$ and $f''(x) \leq c_2$.
 Therefore
 Adler's condition: $A(f,x) := |f''(x)|/|f'(x)|^2 \leq A := c_2/c^2$
 holds for all $x \in \S^1$ (see \cite{Ad,Bow79}).
 This means that Adler's conditions holds uniformly over all iterates.
 Indeed, due to a recursive formula $A(f \circ g,x) = (Af, g(x)) +
 \frac{1}{f' \circ g(x)} \cdot A(g,x)$, we get
 $$
 \frac{|(f^n)''|}{|(f^n)'|^2} \leq \frac{Ac}{c-1} \qquad \text{ for all } n \geq 0.
 $$
Fix $\eta := \frac{c-1}{2cA}$ and abbreviate $K_n(z) = \overline{B(z,e^{-\aa n})}$.
Let
$$
k_n(z) = \max\{ k \geq 0 : \diam( f^k(K_n(z)) ) \leq \eta \}.
$$
Take $\xi \in K_n(z)$ such that $(f^{k_n(z)})'(\xi) |K_n(z)| = |f^{k_n(x)}(K_n(z))| \in [\eta/c^+, \eta]$. Then
$$
\left|\frac{1}{ (f^{k_n(z)})'(\xi)} - \frac{1}{ (f^{k_n(z)})'(z) } \right|
= \left| \int_z^\xi \left( \frac{1}{ (f^{k_n(z)})'(x)} \right)' \, dx \right|
= \int_z^\xi A(f^{k_n(z)}, x) \, dx \leq\frac{Ac}{c-1} |\xi-z|.
$$
This gives
$$
\left| 1 - \left| \frac{ (f^{k_n(z)})'(\xi) }{ (f^{k_n(z)})'(z) } \right| \right|
\leq \frac{Ac}{c-1} \left| (f^{k_n(z)})'(\xi)\right| \,  |\xi-z|
\leq \frac{Ac}{c-1} \eta = \frac12,
$$
so that $\frac12 \leq \left| \frac{ (f^{k_n(z)})'(\xi) }{ (f^{k_n(z)})'(z) } \right| \leq \frac32$.
Recalling that $|K_n(z)| = 2 e^{-\aa n}$, we get
$
\frac{\eta}{3c^+} e^{\aa n} \leq |(f^{k_n(z)})'(z)| \leq \eta e^{\aa n}$, so that
$$
{\aa} + \frac{1}{n} \log \frac{\eta}{3c^+} \leq
\frac{k_n(z)}{n} \frac{1}{k_n(z)} \log |(f^{k_n(z)})'(z)| \leq {\aa} + \frac1n \log \eta.
$$
In the inferior/superior limit, we get $\liminf_{n \to \infty} \frac{k_n(z)}{n} = \frac{\aa}{\overline{\lambda}(z)}$
and $\limsup_{n \to \infty} \frac{k_n(z)}{n} = \frac{\aa}{\underline{\lambda}(z)}$,
respectively.

Let $N_\eta$ be such that $f^{N_\eta}(J) \supset \S^1$ for every interval $J$ of length $\eta$.
Then for the statement on the $\limsup$, we get
\begin{eqnarray*}
\limsup_{n\to\infty} \frac1n \log S(n,\eps,K_n(z)) &\geq&
\limsup_{n\to\infty} \frac{1}{n} \log S\left(n-k_n(z)-N_\eta,\eps,f^{k_n(z) + N_\eta}(K_n(z)) \right) \\
&=& \limsup_{n\to\infty} \frac{n-k_n(z)-N_\eta}{n} \frac{\log S(n-k_n(z)-N_\eta,\eps,\S^1 ) }{n-k_n(z)-N_\eta} \\
&=&  \left(1-\frac{\aa}{\overline\lambda(z)} \right) \limsup_{n\to\infty} \frac{\log S(n-k_n(z)-N_\eta,\eps,\S^1 ) }{n-k_n(z)-N_\eta}. 
\end{eqnarray*}
Taking the limit $\eps \to 0$, we obtain $\overline h_{\aa}(z) \geq 
\left(1-\frac{\aa}{\overline\lambda(z)} \right) \htop(f)$.

For the other inequality, note that restrictions $f^j|_{\overline{B(z,e^{-\aa n})}}$, $1 \leq j \leq k_n(z)$, are homeomorphisms, and therefore it can separate at most $1/\eps$ points per iterate. Therefore
\begin{eqnarray*}
\limsup_{n\to\infty} \frac1n \log S(n,\eps,K_n(z)) &\leq&
\limsup_{n\to\infty} \frac{1}{n} \log \left( S\left(n-k_n(z),\eps,f^{k_n(z)}(K_n(z)\right) \cdot \frac{k_n(z)}{\eps} \right) \\
&\leq& \limsup_{n\to\infty} \frac{n-k_n(z)}{n} \frac{ \log S(n-k_n(z),\eps,\S^1) }{n-k_n(z)} + \frac1n \log \frac{k_n(z)}{\eps}\\
&=&  \left(1-\frac{\aa}{\overline\lambda(z)} \right) \limsup_{n\to\infty} \frac{\log S(n-k_n(z),\eps,\S^1 ) }{n-k_n(z)}. \
\end{eqnarray*}
Taking the limit $\eps \to 0$, we obtain $\overline h_{\aa}(z) \leq 
\left(1-\frac{\aa}{\overline\lambda(z)} \right) \htop(f)$,
so $\overline h_{\aa}(z) = 
\left(1-\frac{\aa}{\overline\lambda(z)} \right) \htop(f)$.

Similarly, for the statement on the $\liminf$, we get
\begin{eqnarray*}
\liminf_{n\to\infty} \frac1n \log S(n,\eps,K_n(z)) &\geq&
\liminf_{n\to\infty} \frac{1}{n} \log S(n-k_n(z)-N_\eta,\eps,f^{k_n(z) + N_\eta}(K_n(z)) \\
&=& \liminf_{n\to\infty} \frac{n-k_n(z)-N_\eta}{n} \frac{\log S(n-k_n(z)-N_\eta,\eps,\S^1 ) }{n-k_n(z)-N_\eta} \\
&=&  \left(1-\frac{\aa}{\underline\lambda(z)} \right) \liminf_{n\to\infty} \frac{\log S(n-k_n(z)-N_\eta,\eps,\S^1 ) }{n-k_n(z)-N_\eta}.
\end{eqnarray*}
Taking the limit $\eps \to 0$, we obtain $\underline h_{\aa}(z) \geq 
\left(1-\frac{\aa}{\underline\lambda(z)} \right) \htop(f)$.

The other inequality is obtained as follows:
\begin{eqnarray*}
\liminf_{n\to\infty} \frac1n \log S(n,\eps,K_n(z)) &\leq&
\liminf_{n\to\infty} \frac{1}{n} \log \left( S\left(n-k_n(z),\eps,f^{k_n(z)}(K_n(z))\right) \cdot \frac{k_n(z)}{\eps} \right) \\
&\leq& \liminf_{n\to\infty} \frac{n-k_n(z)}{n} \frac{ \log S(n-k_n(z),\eps,\S^1) }{n-k_n(z)} + \frac1n \log \frac{k_n(z)}{\eps}\\
&=&  \left(1-\frac{\aa}{\underline\lambda(z)} \right) \liminf_{n\to\infty} \frac{\log S(n-k_n(z),\eps,\S^1 ) }{n-k_n(z)}.
\end{eqnarray*}
Taking the limit $\eps \to 0$, we obtain $\underline h_{\aa}(z) \leq 
\left(1-\frac{\aa}{\underline\lambda(z)} \right) \htop(f)$,
so $\underline h_{\aa}(z) = 
\left(1-\frac{\aa}{\underline\lambda(z)} \right) \htop(f)$.
This finishes the proof.
\end{proof}

\subsection{\Tl entropy for higher-dimensional maps}\label{sec:2D}

\begin{lemma}
 If $F = f \times g:X \times Y \to X \times Y$ is a Cartesian product and $z = (x,y)$, then
 $h_{\aa}(F,z) = h_{\aa}(f,x) + h_{\aa}(g,y)$.
\end{lemma}

\begin{proof}
 We can take Cartesian product of $(n,\eps)$-separated sets as
 $(n,\eps)$-separated sets for the Cartesian product $f \times g$, and there are
 no more efficient sets. Inserting this in the definition of \tl entropy proves the result.
\end{proof}

\begin{theorem}\label{thm:toral}
Let $A$ be a unimodular integer $d \times d$-matrix
and $F_A:\T^d \to \T^d$, $x \mapsto Ax \bmod 1$, the affine toral automorphism determined by $A$,
  then {\red $h_{\aa}(z) = \sum_{\log |\lambda_i| \geq \aa } (\log |\lambda_i| - \aa)$,} where $\lambda_i$, $i = 1, \dots, d$, are the eigenvalues of the matrix $A$.
\end{theorem}

\begin{proof}
 Let $\mu_1 > \mu_2 > \dots > \mu_k$ be the set of absolute values of eigenvalues and $\kappa_j$ is the number of eigenvalues (counted with multiplicity) of absolute value $\mu_j$.
 Each $\mu_j$ represents (possibly several) eigenvalues, and the direct sums of their generalized eigenspaces are denoted as $E_j$; it is a hyperplane and we let $1 \leq d_j \leq \dim(E_j)$ be the size of the largest Jordan block in
 the Jordan representation of $F_A|_{E_j}$.
 If $K_{n,j}(z)$ is the intersection of $B(z, e^{-\aa n})$ and a local plane in the direction
 of $E_j$, then $F_A^k(K_{n,j}(z))$ reaches unit size after $k_n(z) = C n^{1-d_j} e^{ n \log \mu_j}$ iterates, for some uniform constant $C> 0$.
 Here we assume that $\log \mu_j \geq \aa$, because otherwise
$F_A^k(K_{n,j}(z))$ does not reach unit size for $k \leq n$ iterates.
 Therefore,  $S(n,\eps, K_{n,j}(z)) \approx S(n-k_n(z), \eps, E_j(z))$, where $E_j(z)$ is a unit size $\kappa_j$-dimensional ball centered at $z$ in the direction of $E_j$.
 Taking the logarithm and the limit $n \to \infty$, we obtain $h_\aa(z) = \lim_{n\to\infty} \frac{n-k_n(z)}{n} \frac{1}{n-k_n(z)}\log S(n-k_n(z), \eps, E_j(z)) = \kappa_j \max\{ \log \mu_j - \aa, 0\}$.
 Locally, $F_A$ acts as the Cartesian product of $F_A|_{E_1(z)} \times \dots \times F_A|_{E_k(z)}$.
 Therefore $h_\aa(z) = \sum_{\log \mu_j \geq \aa} (\log \mu_j - \aa) =
 \sum_{\log |\lambda_i| \geq \aa} (\log |\lambda_i| - \aa)$, as claimed.
\end{proof}

In fact, if $F$ is a toral automorphism (or Anosov map) that is not necessarily affine,
but has bounded distortion, then the formula becomes
$\overline{h}_{\aa}(z) = \htop(f) - \sum_{\log |\overline{\lambda}_i(z)| \geq \aa}
(\log|\overline{\lambda}_i(z)|-\aa)$,
where $\overline{\lambda}_i$ stands for the $i$-th upper Lyapunov exponent.
The same result works for affine hyperbolic horseshoe maps (or non-affine but with bounded distortion).
The global product structure of these maps is important to get this outcome, because the formula
can be quite different too, as the following example shows.

\begin{example}\label{ex:disk}
Let $\overline \D$ be the closed unit disk, and let $f:\overline \D \to \overline \D$ be given in polar coordinates by $f(r,\phi) \mapsto (r(2-r), 3\phi \pmod {2\pi})$. The derivative of $r \mapsto r(2-r)$ at $r = 0$ is $2$.
Therefore it takes $K_n(z) := \overline{B(z, e^{-\aa n})}$ about $\frac{\aa n + \log \eps}{\log 2}$ iterates to reach size $\eps$, and at this point, the $\phi$-component will let $S(n,\eps,K_n(z))$
grow, with exponential rate $\log 3$.
It follows that $h_{\aa}(0,0) = (1 - \frac{\aa}{\log 2})\log 3$.
\end{example}

It is maybe good at this point to make a comparison with
the notion of local entropy of $f:X \to X$ used by Buzzi \& Ruette \cite{Buz97,BR} and Newhouse \cite{N89}.
In their definition, an $(n,\eps)$-Bowen ball takes the place of our exponentially small ball $B(z,e^{-\aa n})$.
Thus the {\em local entropy} is {\blue
\begin{equation}\label{eq:Buz}
 \hloc(f) = \lim_{\eps \to 0} \hloc(f,\eps)
 \quad \text{ for } \quad
\hloc(f,\eps) = \lim_{\delta \to 0}\limsup_{n\to\infty} \frac1n \sup_{x \in X} \log r(n,\delta, B_n(x,\eps)).
\end{equation}
}

In fact, Buzzi uses {\blue$(n,\delta)$-spanning sets
(the minimal cardinality of which is denoted by $r$ in~\eqref{eq:Buz}) instead of $(n,\delta)$-separating sets, }  but that is of no consequence. Also, by taking the supremum over $x \in X$,
$\hloc(f)$ becomes a quantity independent of the point in space.
Buzzi shows that $\hloc(f) \leq \frac{\dim M}{r} \log \inf_n \root{n}\of{\| Df^n\|_\infty}$ if $f:M \to M$ is a $C^r$-smooth map on a manifold $M$, and hence
$\hloc(f) = 0$ for $C^\infty$ maps.
But if we adjust Example~\ref{ex:disk} to {\blue$f(r,\phi) = (r,3\phi \pmod {2\pi})$,} then the point $(0,0)$ gives
local entropy  $\log 3$.

Compared to $h_{\aa}(z)$, for conformal maps (such as the $C^2$ expanding circle maps of Theorem~\ref{thm:circle}), we find $\hloc(f) = h_{\aa}(z) = 0$ for $\aa = \lambda(z)$, but for non-conformal maps we don't expect $\hloc(f)$ and $h_{\aa}(z) = 0$ to coincide.

\section{Local measure-theoretic entropy and 
pressure}\label{sec:localpressure}

\subsection{Brin-Katok local measure entropy}
\label{sec:BK}
For a continuous map $f:X\to X$ defined on a compact metric space $(X,d)$ and $f$-invariant Borel probability measure $\mu$, Brin \& Katok \cite{BK83}
introduced notions of {\em local lower} and {\em local upper}
{\em measure-theoretic entropy} w.r.t.\ $\mu$ as follows
\begin{equation}\label{eq:BK}
\begin{cases}
{\underline h}_{\mu}(f,x):=\lim_{\eps \to 0} \liminf_{n \to \infty} -\frac{1}{n} \log \mu(B_n(x,\eps)), \\[2mm]
{\overline h}_{\mu}(f,x):=\lim_{\eps \to 0} \limsup_{n \to \infty} -\frac{1}{n} \log \mu(B_n(x,\eps)).
\end{cases}
\end{equation}

\begin{theorem}[Main Theorem and Corollary of \cite{BK83}]\label{thm:BK}
For a continuous map $f: X\to X$ on a compact metric space $(X,d)$ and Borel $f$-invariant probability measure $\mu$, the equality ${\underline h}_{\mu}(f,x)= {\overline h}_{\mu}(f,x):=h_{\mu}(f,x)$ holds
for $\mu$-a.e.\ $x \in X$. Moreover, $h_{\mu}(f,x)$ is $f$-invariant and
$$
\int_X h_{\mu}(f,x)\, d\mu(x) = h_{\mu}(f).
$$
{\red If $\mu$ is ergodic, then additionally $h_{\mu}(f,x) = h_\mu(f)$
for $\mu$-a.e.\ $x$.}
\end{theorem}

Feng \& Huang \cite[Theorem 1.2(i)]{FH12}
showed in 2012 the following variational principle for the Brin-Katok  local entropy.  For a continuous map $f:X\to X$ defined on a compact metric space $(X,d)$ and a
 compact non-empty subset $K \subset X$ they obtain 

\begin{equation}\label{eq:FH}
\htop^{FH}(f,K) = \sup\{ \underline{h}_\mu(f) : \mu \text{ \red is a Borel probability measure and } \mu(K)=1\},
\end{equation}
{\red where $\underline{h}_\mu(f)$ local lower entropy.
However, here $\htop^{FH}(f,K)$ is   an entropy-like quantity constructed in a way similar to the construction of $s$-Hausdorff measure and Hausdorff dimension,
 called {\em Bowen topological entropy}, see \cite[Section 2.2]{FH12}. 
This is based on a Carath\'eodory-like construction,
but a bit more involved than the Bowen topological entropy used
in Theorem~\ref{thm:MW} below, and which we discuss in the next section, generalized to pressure.}

\subsection{Topological pressure and Carath\'eodory structures}\label{sec:toppressure}

In 2008, Ma \& Wen \cite{MW} noticed relationship between topological entropy restricted to a set $Z \subset X$ and  local measure entropies on $Z$.

\begin{theorem}[Theorem 1 in \cite{MW}] \label{thm:MW}
For a continuous map $f: X\to X$ on a compact metric space $(X,d)$, a Borel  probability measure $\mu$,
a Borel subset $Z \subset X$ and $s \geq 0$ we have:
\begin{enumerate}
\item If ${\underline h}_{\mu}(f,x) \le s$~for~all~$x \in X$, then $\htop^B(f,Z) \le s$.
\item ${\overline h}_{\mu}(f,x) \ge s$~for~all~$x \in X$ and $\mu(Z)>0$, then $\htop^B(f,Z) \ge s$.
\end{enumerate}
\end{theorem}

{\red
The Bowen entropy $\htop^B$ used in this result is based on Carath\'eodory construction which we will discuss in this subsection at length, but for topological pressure, so with the inclusion of a potential function $\phi:X \to \R$.
By setting $\phi \equiv 0$ in the topological pressure defined below, we get $\htop^B(f)$. Wang \& Zhang \cite{WZ25} obtain a similar result for upper local entropies $\overline h_\mu(f, x)$
and {\em packing} topological entropies.

The aim of this subsection is to} provide a generalization of Theorem~\ref{thm:MW} to topological pressure.
Topological pressure, in the spirit of Carath\'eodory structures elaborated by Pesin \cite{Pes}, for a continuous map $f:X \to X$ defined on a compact metric space $(X,d)$, and restricted to a subset $Z \subset X$.
Fix a continuous map (called {\em potential}) $\phi:X \to \R$, $N \in \N$, $r>0$ and $s\ge 0$, define
$$
M_Z(s,r,\phi,N):=\inf\left\{\sum_{j \in J} \exp[-s\cdot n_j +\sum_{m=0}^{n_j-1}\phi(f^m(x_j))]: Z \subset \bigcup_{j \in J} B_{n_j}(x_j, r);~n_j \ge N\right\},
$$
where the infimum is over all finite or countable covers of $Z$ by Bowen balls $\{ B_{n_j}(x_j, r)\}_{j \in J}$ with $n_j \ge N$.
Here, the Bowen balls are indexed by the elements of a finite or countable set $J$.
Using standard arguments (see \cite{Pes})
one can easily show that for any $N \in \N$, the inequality $M_Z(s,r,\phi,N+1) \ge M_Z(s,r,\phi,N)$ holds. Therefore, there exists a limit
$$
M_Z(s,r,\phi):= \lim_{N \to \infty} M_Z(s,r,\phi,N).
$$
The graph of the function $s \mapsto M_Z(s,r,\phi)$ is very similar to the graph of $s$-Hausdorff measure function, i.e., there exists a unique critical parameter $s_0$, where the function
$s \mapsto M_Z(s,r,\phi)$ drops from $\infty$ to $0$. Thus, we can define
$$
M_Z(r,\phi):= \sup \{s \ge 0:  M_Z(s,r,\phi)=\infty   \}= \inf\{s \ge 0: M_Z(s,r,\phi) =0\}=s_0.
$$
Fix $r_1 \le r_2$ and consider a cover $\{ B_{n_j}(x_j, r_1)\}_{j \in J}$ of $Z$, with $n_j \ge N$. Then $\{ B_{n_j}(x_j, r_2)\}_{j \in J}$ is a cover of $Z$,
so $M_Z(s, r_1,\phi) \ge M_Z(s, r_2,\phi)$. Therefore
$$
M_Z(r_1,\phi)= \inf\{s \ge 0: M_Z(s,r_1,\phi) =0\}\ge \inf\{s \ge 0: M_Z(s,r_2,\phi) =0\} =M_Z(r_2,\phi).
$$
This proves that the function $r \to M_Z(r,\phi)$ is non-increasing, so there exists a limit
$$
P_Z(f,\phi):=\lim_{r \to 0} M_Z(r,\phi).
$$
The quantity $P_Z(f,\phi)$ is called the {\it topological pressure} of $f$, restricted to $Z$, with respect to the potential $\phi: X \to \R$.
Basic properties of the topological pressure are presented in the following lemma.

\begin{lemma}\label{lem:basprop}
For a continuous map $f:X \to X$ defined on a compact metric space $(X,d)$, $r>0$, $s\ge 0$, $N \in \N$ and a continuous potential $\phi: X \to \R$ we have:
\begin{enumerate}
\item If $Z_1\subset Z_2 \subset X$, then $M_{Z_1}(s,r,\phi,N)\le M_{Z_2}(s,r,\phi,N)$.
 \item If $Z_1\subset Z_2 \subset X$, then $P_{Z_1}(f, \phi) \le P_{Z_2}(f, \phi)$.
 \item If $Z =\bigcup_{k \in \N}Z_k$, then $P_{Z}(f, \phi) =\sup\{P_{Z_k}(f, \phi): k \in \N \}$.
\end{enumerate}
\end{lemma}
\bigskip
Next, we give a modification of the classical covering lemma (see Section 2.8.4 in \cite{Fed}).

\begin{lemma}\label{lem:cover} [Lemma 1 in \cite{MW}]
Let $f:X \to X$ be a continuous map on a compact metric space. Let $r>0$  and $\cB(r)=\{B_n(x,r): x \in X, n \in \N\}$. For any family $\cF\subset \cB(r)$ there exists  a subfamily
$\cG \subset \cF$ consisting of disjoint Bowen balls such that
$$
\bigcup_{B_n(x,r) \in \cF} B_n(x,r) \subset \bigcup_{B_n(x,r) \in \cG} B_n(x,3 \cdot r).
$$
\end{lemma}

\subsection{Local measure-theoretic pressures} \label{sec:locpressure}
For a continuous map $f:X \to X$ defined on a compact metric space $(X,d)$, a Borel probability measure $\mu$ defined on $X$,
a continuous potential $\phi:X \to \R$ and $x \in X$, we define the {\em upper local measure-theoretic pressure} of $f$ at $x$, with respect to $\mu$,
by
\begin{equation}\label{eq:locpres}
{\overline P}_{\mu}(f,\phi,x):= \lim_{\eps \to 0} \limsup_{n \to \infty} \frac{1}{n} \left(  \sum_{m=0}^{n-1}\phi(f^m(x)) - \log(\mu(B_n(x,\eps))) \right).
\end{equation}
Similarly, we define the {\em lower local measure-theoretic pressure} of $f$ at $x$, with respect to $\mu$, by
$$
{\underline P}_{\mu}(f,\phi,x):= \lim_{\eps \to 0} \liminf_{n \to \infty}\frac{1}{n}
\left( \sum_{m=0}^{n-1}\phi(f^m(x)) - \log(\mu(B_n(x,\eps))) \right).
$$

\begin{remark}\label{rem:CP}
Notice that local measure-theoretic pressures, calculated with respect to the potential $\phi \equiv 0$, coincide with local measure-theoretic Brin-Katok entropies. Carvalho \& P\'erez \cite{CP} showed that {\red if $\mu$ is ergodic,} then ${\overline P}_{\mu}(f,\phi,x) = {\underline P}_{\mu}(f,\phi,x)$ for $\mu$-a.e.\ $x$, and that the Brin-Katok Theorem~\ref{thm:BK} extends to local pressure:
$$
\int_X {\overline P}_{\mu}(f,\phi,x) \, d\mu =
P_\mu(f,\phi) := h_\mu(f) + \int_X \phi \, d\mu.
$$
\end{remark}

We have the following partial variational principle,
extending Ma \& Wen's result Theorem~\ref{thm:MW}.

\begin{theorem} \label{thm:aaMW}
Given a continuous map $f:X \to X$ on a compact metric space $(X,d)$, a Borel subset $Z \subset X$ and a Borel probability measure $\mu$ on $X$, we have
for any $s\ge 0$ and continuous $\phi:X \to \R$:

1) If ${\overline P}_{\mu}(f,\phi,x)\le s$ for all $x \in Z$, then $P_Z(f,\phi) \le s$.

2) If $\mu(Z)>0$ and ${\underline P}_{\mu}(f,\phi,x)\ge s$ for all $x \in Z$, then $P_Z(f,\phi) \ge s$.
\end{theorem}

\begin{proof} (1)  Fix a Borel subset $Z \subset X$, a Borel probability measure $\mu$, a continuous potential $\phi:X \to \R$ and $\eps >0$. Assume that there exists $s\ge 0$ such that $ {\overline P}_{\mu}(f,\phi,x)\le s$ for every $x \in Z$. Define a sequence of sets $(Z_k)_{k \in \N}$ by
$$
Z_k:=\left\{x \in Z: \limsup_{n \to \infty} \frac{1}{n}
\left( \sum_{m=0}^{n-1}\phi(f^m(x)) - \log(\mu(B_n(x,r))) \right) \le s+\eps \text{ for all } 0 < r < \frac{1}{3k} \right\}.
$$
Then $Z=\bigcup_{k \in \N} Z_k$. Now we fix $k \in \N$ and $r\in (0,\frac{1}{3k})$. The definition of $Z_k$ yields that for any $x \in Z_k$, there exists a strictly increasing sequence $(n_j(x))_{j \in \N}$ with
$$
\sum_{m=0}^{{n_j(x)}-1}\phi(f^m(x))
- \log(\mu(B_{n_j(x)}(x,r)))  \le (s+\eps)\cdot n_j(x).
$$
Notice that $Z_k$ is contained in the union of elements of the family 
$$
\cF_N:=\{B_{n_j(x)}(x,r)):x \in Z_k,~n_j(x)\ge N, j \in J\}.
$$
By Lemma~\ref{lem:cover}, there exists a subfamily $\cG_N$ of $\cF_N$, by pairwise disjoint Bowen balls,
$$
\cG_N:=\{B_{n_i(x_i)}(x_i,r)):x_i \in Z_k,~n_i(x)\ge N, i \in I\},
$$
where $I \subset J$ such that
$Z_k \subset \bigcup_{i \in I} B_{n_i(x_i)}(x_i,3\cdot r)$
and 
$$
\mu(B_{n_i(x_i)}(x_i,r)) \ge   \exp\left(   -(s+\eps)n_i(x_i)          +  \sum_{m=0}^{{n_i(x_i)}-1}\phi(f^m(x_i))  \right).
$$
Therefore,
\begin{eqnarray*}
M_{Z_k}(s+\eps,3\cdot r, \phi, N) &\le& \sum_{i\in I}  \exp\left(  -(s+\eps)n_i(x_i)  +  \sum_{m=0}^{{n_i(x_i)}-1}\phi(f^m(x_i)) \right) \\
&\le& \sum_{i\in I} \mu(B_{n_i(x_i)}(x_i,r)))\le 1,
\end{eqnarray*}
where the last inequality follows by the disjointness of Bowen balls in the family $\cG_N$.
Passing to the limit $N \to \infty$,
we get
$$
M_{Z_k}(s+\eps, 3\cdot r, \phi) \le 1.
$$
This means that $M_{Z_k}( 3\cdot r, \phi)\le s+\eps$. As $r \to 0$, we obtain $P_{Z_k}(f,\phi) \le s+\eps$.
Finally, by Lemma~\ref{lem:basprop}, we obtain
$$
P_{Z}(f,\phi)=\sup\{P_{Z_k}(f,\phi): k \in \N \}    \le s+\eps,
$$
and consequently
$$
P_{Z}(f,\phi)=\sup\{P_{Z_k}(f,\phi): k \in \N \} \le s,
$$
since $\eps>0$ is arbitrarily small.

\bigskip
(2) Fix $Z \subset X$ such that $\mu(Z)>0$ and $s\ge 0$. Assume that ${\underline P}_{\mu}(f,\phi,x) \ge s$ for all $x \in Z$.
Fix $\eps >0$ and for any $k \in \N$ define
$$
V_k:=\left \{x \in Z: \liminf_{n \to \infty} \frac{1}{n}\left( \sum_{m=0}^{n-1}\phi(f^m(x)) - \log(\mu(B_n(x,r))) \right) \ge s-\eps, \text{ for every } 0 < r < \frac{1}{k} \right \}.
 $$
Since the sequence of sets $(V_k)_{k \in \N}$ increases to $Z$, by continuity of the Borel probability measure $\mu$ we get
$$
\lim_{k \to \infty} \mu(V_k)=\mu(Z) >0.
$$
There exists $k_0 \in \N$ such that $\mu(V_{k_0}) > \frac{1}{2} \mu(Z)>0$.
Now define a sequence of sets $(V_{k_0, N})_{N \in \N}$ by
$$
V_{k_0, N}:=\left\{x \in Z:  \frac{1}{n}\left(  \sum_{m=0}^{n-1}\phi(f^m(x))- \log(\mu(B_n(x,r))) \right) \ge s-\eps, \text{ for all } n \ge N, 0 < r < \frac{1}{k_0}  \right\}.
$$
Notice that the sequence $(V_{k_0, N})_{N \in \N}$ increases to $V_{k_0}$, so by continuity of the Borel probability measure $\mu$, there exists $N_0 \in \N$ such that $\mu(V_{k_0, N_0}) > \frac{1}{2}\mu(V_{k_0})>0$.
By definition of $V_{k_0, N_0}$, for any $x \in V_{k_0, N_0}$, $n_j \ge N_0$ and $r \in(0, \frac{1}{k_0})$, we have
$$
\sum_{m=0}^{n_j-1}\phi(f^m(x)) - \log(\mu(B_{n_j}(x,r))) > (s-\eps) \cdot n_j,
$$
so
$$
\mu(B_{n_j}(x,r)) < \exp\left[ -(s-\eps) \cdot n_j + \sum_{m=0}^{n_j-1}\phi(f^m(x))\right].
$$
Next, consider a countable cover $\cC$ of $V_{k_0, N_0}$ by Bowen balls, defined by
$$
\cC:=\left\{   B_{n_j}(y_j, r): y_j \in V_{k_0, N_0},~n_j \ge N_0,~r \in (0, \frac{1}{k_0}),~B_{n_j}(y_j, r) \cap V_{k_0, N_0} \neq \emptyset, j\in J  \right\}.
$$
Notice that for any $\delta \in (0,\frac{1}{2}\mu(V_{k_0, N_0}))$, there exists a countable family $J_{\delta}$ such that
\begin{eqnarray*}
M_{V_{k_0, N_0}}(s-\eps, r,\phi,N)
&\ge& -\delta + \sum_{j \in J_{\delta}}  \exp\left( -(s-\eps) \cdot n_j  + \sum_{m=0}^{n_j-1}\phi(f^m(x_j)) \right) \\
&\ge& -\delta+ \sum_{j \in J_{\delta}} \mu(B_{n_j}(y_j, r) )\ge -\delta +\mu(V_{k_0, N_0}) \ge \frac{1}{2}\mu(V_{k_0, N_0}).
\end{eqnarray*}
Since $V_{k_0, N_0} \subset Z$, for any potential $\phi: X \to \R$ and $N \in \N$, Lemma~\ref{lem:basprop} gives
$$
M_Z(s-\eps, r,\phi,N) \ge M_{V_{k_0, N_0}}(s-\eps, r,\phi,N)  \ge \frac{1}{2}\mu(V_{k_0, N_0})>0.
$$
Passing to the limit $N \to \infty$, we get $M_Z(s-\eps, r,\phi)>0$.
This means that $M_Z( r,\phi)>s-\eps$.
Finally, passing to the limit $r\to 0$ and taking into account that $\eps$ is arbitrarily small, we obtain
$P_Z(f,\phi)\ge s$. The proof is complete.
\end{proof}

\subsection{The \tl version of measure-theoretic local pressure}
\label{sec:aapressure}
{\blue In this section we introduce a  \tl version of the upper measure-theoretic local pressure, which is a modification of ${\overline P}_{\mu}(f,\phi,z)$ from~\eqref{eq:locpres},
in the following way. }
\begin{equation*}\label{eq:aapres}
\overline{P}_{\mu,\aa}(f,\phi, z) := \limsup_{n\to\infty} \frac1n \left( \sum_{m=0}^{n-1} \phi(f^m(z))- \log \mu(B(z,e^{-\aa n})) \right),
\end{equation*}
and an analogous replacement for lower measure-theoretic local pressure.
\begin{equation*}\label{eq:aapres2}
\underline{P}_{\mu,\aa}(f,\phi, z) := \liminf_{n\to\infty} \frac1n \left( \sum_{m=0}^{n-1} \phi(f^m(z))- \log \mu(B(z,e^{-\aa n})) \right).
\end{equation*}
The topological \tl pressure itself we will define in the spirit of \cite{Pes}.
That is, define
$$
M_{f,Z,\aa}(s,\phi,N):=\inf\left\{\sum_{j \in J} \exp[-s\cdot n_j + \sum_{m=0}^{n_j-1}\phi(f^m(x_j))]: Z \subset \bigcup_{j \in J} B(x_j, e^{-\aa n_j});~n_j \ge N\right\},
$$
where the infimum is over all finite or countable covers of $Z$ by metric balls $\{ B(x_j, e^{-\aa n_j})\}_{j \in J}$ with $n_j \ge N$.
Let
$$
\overline{M}_{f,Z,\aa}(s,\phi) := \limsup_{N \to \infty} M_{f,Z,\aa}(s,\phi,N),
$$
$$
\underline{M}_{f,Z,\aa}(s,\phi) := \liminf_{N \to \infty} M_{f,Z,\aa}(s,\phi,N).
$$

\begin{lemma}
The graph of the function $s \mapsto \overline{M}_{f,Z,\aa}(s,\phi)$ has a unique critical parameter $s_0$, where the function
$s \mapsto \overline{M}_{f,Z,\aa}(s,r,\phi)$ drops from $\infty$ to $0$.
The graph of the function $s \mapsto \underline{M}_{f,Z,\aa}(s,r,\phi)$ has a unique critical parameter $s_1$, where the function
$s \mapsto \underline{M}_{f,Z,\aa}(s,\phi)$ drops from $\infty$ to $0$.
\end{lemma}

\begin{proof} Proof of the lemma is similar to the proof of Proposition 1.2 in \cite{Pes}.
\end{proof}

Thus, we define
$$
\overline{P}_{Z,\aa}(f, \phi):= \sup \{s \ge 0:  \overline{M}_{f,Z,\aa}(s,\phi)=\infty \}
= \inf\{s \ge 0: \overline{M}_{f,Z,\aa}(s,\phi) =0\} = s_0,
$$
$$
\underline{P}_{Z,\aa}(f, \phi):= \sup \{s \ge 0:  \underline{M}_{f,Z,\aa}(s,\phi)=\infty \}
= \inf\{s \ge 0: \underline{M}_{f,Z,\aa}(s,\phi) =0\} = s_1.
$$

A metric space $(X,d)$ is called boundedly compact if all bounded closed subsets of X are compact. In particular $\R^n$ and Riemannian manifolds (see \cite[p.9]{Gro07}) are boundedly compact.

\begin{lemma}\label{lem:cover2}(Vitali Covering Lemma, Theorem 2.1 in \cite{Ma95}).  Let $(X,d)$ be a boundedly compact metric space and  a family of closed balls $\cB= \{ \overline{B(x,r)} : x \in X, r>0\}$ in X such that
$$
\sup \{diam(B(x,r)) : B(x,r) \in \cB  \} < \infty.
$$
Then, there is a finite or countable sequence $B(x_i,r_i) \in \cB$, indexed by $i \in I$, of disjoint balls such that
$$
\bigcup_{B(x,r) \in \cB} B(x,r) \subset \bigcup_{i \in I}
\overline{ B(x_i,5\cdot r_i) }.
$$

\end{lemma}
Properties analogous to those of Lemma~\ref{lem:basprop} hold for this version of topological pressure.

\begin{theorem}\label{thm:tlMW}
Given a continuous map $f:X \to X$ on a compact metric space $(X,d)$, a Borel subset $Z \subset X$ and a Borel probability measure $\mu$ on $X$, we have
for any $s\ge 0$ and continuous $\phi:X \to \R$:

1) If $ {\overline P}_{\mu,\aa}(f,\phi,x) \le s$ for all $x \in Z$, then $\overline{P}_{Z,\aa}(f,\phi) \le s$.
{\blue

2) If $\mu(Z)>0$ and ${\underline P}_{\mu,\aa}(f,\phi,x)\ge s$ for all $x \in Z$, then $\underline{P}_{Z,\aa}(f,\phi) \ge s$. }
\end{theorem}

\begin{proof} (1)  Fix a Borel subset $Z \subset X$, a Borel probability measure $\mu$, a continuous potential $\phi:X \to \R$ and $\eps >0$. Assume that there exists $s\ge 0$ such that $ {\overline P}_{\mu,\aa}(f,\phi,x)\le s$ for every $x \in Z$.
Therefore, for $x \in Z$, there exists a strictly increasing sequence $(n_j(x))_{j \in \N}$ with
$$
\sum_{m=0}^{{n_j(x)}-1}\phi(f^m(x))
- \log(\mu(B(x,e^{-\aa n_j(x)})))  \le (s+\eps)\cdot n_j(x).
$$
Notice that $Z$ is contained in the union of elements of the family
$$
\cF_N:=\{B(x_j,e^{-\aa n_j(x_j)}):x_j \in Z,~n_j(x_j)\ge N, j \in J\}.
$$
Choose $n_0$ minimal such that $e^{\aa n_0} \geq 5$.
By Lemma~\ref{lem:cover2} there exists a subfamily $\cG_N$ of $\cF_N$, indexed by elements of a set $I$, by pairwise disjoint metric balls,
$$
\cG_N:=\{B(x_i,e^{-\aa n_i(x_i)})):x_i \in Z,~n_i(x_i)\ge N, i \in I\},
$$
where $I \subset J$, such that
$Z \subset \bigcup_{i \in I} B(x_i,e^{-\aa (n_i(x_i)-n_0)}))$
and
$$
\mu(B(x_i,e^{-\aa n_i(x_i)}))) \ge  \exp\left(  -(s+\eps)n_i(x_i)          +  \sum_{m=0}^{{n_i(x_i)}-1}\phi(f^m(x_i))  \right).
$$
Therefore,
\begin{eqnarray*}
M_{Z,\aa}(s+\eps, \phi, N+n_0) &\le& \sum_{i\in I}  \exp\left(  -(s+\eps)[n_i(x_i) - n_0 ]+  \sum_{m=0}^{{n_i(x_i)}-n_0-1}\phi(f^m(x_i)) \right) \\
&\le&\exp[(s+\eps)\cdot n_0] \sum_{i\in I}  \exp\left(  -(s+\eps)n_i(x_i)  +  \sum_{m=0}^{{n_i(x_i)}-1}\phi(f^m(x_i)) \right) \\
&\le&\exp[(s+\eps)\cdot n_0]  \sum_{i\in I} \mu(B(x_i,e^{-\aa n_i(x_i))} ))\le \exp[(s+\eps)\cdot n_0],
\end{eqnarray*}
where the last inequality follows by the disjointness of metric balls in the family $\cG_N$.
Taking the $\limsup$ as $N \to \infty$,
we get
$$
\overline{M}_{Z,\aa}(s+\eps, \phi) \le \exp[(s+\eps)\cdot n_0] <\infty
$$
and therefore
$$
\overline{P}_{Z,\aa}(f,\phi) \le s+\eps.
$$
Since $\eps>0$ is arbitrarily small, $\overline{P}_{Z,\aa}(f,\phi) \le s$.

\bigskip
(2) Fix $Z \subset X$ such that $\mu(Z)>0$ and $s\ge 0$. Assume that ${\underline P}_{\mu,\aa}(f,\phi,x) \ge s$ for all $x \in Z$.
Fix $\eps >0$ and for any $N \in \N$, define
$$
V_N := \left\{ x \in Z:  \frac{1}{n}\left(  \sum_{m=0}^{n-1}\phi(f^m(x))- \log(\mu(B(x,e^{-\aa n}))) \right) \ge s-\eps, \text{ for all } n \ge N\right\}.
$$
Notice that the sequence $(V_N)_{N \in \N}$ increases to $Z$, so by continuity of the Borel probability measure $\mu$ there exists $N_0 \in \N$ such that $\mu(V_{N_0}) > \frac{1}{2}\mu(Z)>0$.
By definition of $V_{N_0}$, for any $x \in V_{N_0}$ and $n_j \ge N_0$, we have
$$
\sum_{m=0}^{n_j-1}\phi(f^m(x)) - \log(\mu(B(x_j,e^{-\aa n_j}))) > (s-\eps) \cdot n_j,
$$
so
$$
\mu(B(x,e^{-\aa n_j}))) < \exp\left[ -(s-\eps) \cdot n_j + \sum_{m=0}^{n_j-1}\phi(f^m(x))\right].
$$
Next, consider a countable cover $\cC$ of   $V_{N_0}$ by metric balls, defined by
$$
\cC:=\left\{   B(y_j, e^{-\aa n_j}): y_j \in V_{N_0},~n_j \ge N_0,~B(y_j, e^{-\aa n_j}) \cap V_{N_0} \neq \emptyset, j\in J  \right\}.
$$
Notice that for any $\delta \in (0,\frac{1}{2}\mu(V_{N_0}))$, there exists a countable family $J_{\delta}$ such that
\begin{eqnarray*}
M_{V_{N_0},\aa}(s-\eps, \phi,N)
&\ge& -\delta + \sum_{j \in J_{\delta}}  \exp\left( -(s-\eps) \cdot n_j  + \sum_{m=0}^{n_j-1}\phi(f^m(x_j)) \right) \\
&\ge& -\delta + \sum_{j \in J_{\delta}} \mu(B(x_j,e^{-\aa n_j})) )\ge -\delta +\mu(V_{N_0}) \ge \frac{1}{2}\mu(V_{N_0}).
\end{eqnarray*}
Since $V_{k_0, N_0} \subset Z$, for any potential $\phi: X \to \R$ and $N \in \N$, we get
$$
M_{Z,\aa}(s-\eps, \phi,N) \ge M_{V_{N_0}}(s-\eps,\phi,N)  \ge \frac{1}{2}\mu(V_{N_0})>0.
$$
Taking the $\liminf$ as $N \to \infty$, we get $\underline{M}_{Z,\aa}(s-\eps, \phi)>0$.
This means that $\underline{M}_{Z,\aa}(\phi)>s-\eps$.
Finally taking into account that $\eps$ is arbitrarily small, we obtain
$\underline{P}_{Z,\aa}(f,\phi)\ge s$. The proof is complete.
\end{proof}

\section{Discussion on local entropy}\label{sec:discussion}

In the literature on dynamical systems there are several attempts to
define local entropy. Bowen \cite{Bow} introduced the notion of $\eps$-entropy to get an upper estimate of the difference
between $h_{\mu}(f)$ and the entropy of a partition  with diameter less than $\eps$.
Another notion of local entropy (called conditional topological entropy) was given by Misiurewicz \cite{Mis}, who proved that vanishing of his local entropy implies existence of maximal measure. Finding (uniqueness of) measures of maximal entropy was also
the reason for Newhouse \cite{N89} and Buzzi \& Ruette \cite{Buz97,  BR} to define a version the local entropy $\hloc(f)$, see~\eqref{eq:Buz}.
Buzzi \& Ruette \cite{BR} proved that for
$C^1$ maps $f:[0,1]\to [0,1]$ with critical set $\Crit(f) = \{ x \in [0,1]: f|_U \text{ is not monotone on each neighborhood } U \owns x\}$,
if $\htop(f) > \htop (f, \Crit(f))+\hloc(f)$, then $f$ admits a measure of maximal entropy.

\subsection{Orbit complexity}
 Brudno \cite{Br82} introduced the notion of orbit complexity for a homeomorphism $f:X \to X$ of a compact space $X$ and related it to measure-theoretic entropy and topological entropy. In Brudno's approach the orbit complexity is a measure of the amount of information that is necessary to describe the orbit. The complexity $K(x,f)$ of the orbit of $x \in X$ is defined in the spirit of asymptotically optimal functions in the sense of Kolmogorov; it indicates the limit proportion of $n$-prefixes
 of the symbolic itinerary of $x$ needed to reconstruct the whole $n$-prefix. 
 In his Theorem 3.1, which slightly predates what we state as the Brin-Katok Theorem~\ref{thm:BK} in Section~\ref{sec:BK}, Brudno proved the following result: If an $f$-invariant measure $\mu$ is ergodic, then\footnote{Entropy is computed with respect to logarithmic base $b$ where $b$ is the cardinality of the alphabet used for the symbolic itineraries.} $K(x,f)=h_{\mu}(f)$ for $\mu$-a.e.\ $x \in X$. Theorem 3.2 in \cite{Br82} says that for any point $x\in X$ one has $K(x,f)\leq h_f(x)$, where $h_f(x) := \sup\{ h_\mu(f) : \mu \text{ is a weak accumulation point of } (\frac1n \sum_{k=0}^{n-1} \delta_{f^k(x)})_{n\geq 1} \}$ is the entropy of Kamae introduced in \cite{Ka73}.

The results of Brudno were extended by White \cite{Wh93}, who proved that for any ergodic $f$-invariant measure $\mu$, $\mu$-a.e.\ orbit has a limiting complexity equal to $\htop(f)$ Also, if $f:X \to X$ is minimal, then the set of points with orbit upper complexity equal to the topological entropy is a residual set. If $f:X \to X$ has a unique measure $\mu_0$ of maximal entropy, then any point with lower complexity equal to $\htop(f)$, is generic for the measure $\mu_0$. This is likely to hold also if each measure of maximal entropy is fully supported as in {\red Corollary~\ref{cor:hconstant}.}

Another extension of results of Brudno \cite{Br82} one can find in the paper by Galatolo \cite{Ga00}, who considered a dynamical system given by a continuous map 
on a separable metric space, and introduced a notion of computable structure  on a metric space with a new notion of orbit complexity with respect to this structure. For compact spaces, Galatolo's definition coincides with the definition given by Brudno.

A different way of using single orbits to estimate entropy goes back to time series analysis, and is known as Takens' estimator in applied dynamics and fractal dimension estimation, see e.g.\ \cite{Takens, Theiler}. Viewing orbits $Y := (f^n(x_0))_{n\geq 0}$
for $x_0 \in X$
as time series, one can define $k \times \dim(X)$-dimensional reconstruction vectors
$(x_n, x_{n+1}, \dots , x_{n+k-1})$ and estimate the probability
$\bP_\eps^k(Y)$ of finding two such vectors at mutual distance $\leq \eps$.
The entropy is the logarithmic growth rate of this quantity as $\eps \to 0$:
$h(Y) = \lim_{\eps\to 0} \lim_{k \to \infty} \frac1k \log \bP^k_\eps(Y)$.
Takens' estimator is a maximum {\blue likehood }estimator of this quantity,
which is also applicable to arbitrary time series, i.e., not just for orbits only.

\subsection{Entropy pairs and tuples}
In 1993, Blanchard \cite{Bl93} introduced the notion of entropy pairs, in order to localize ``where''  in the system entropy is generated.
This was generalized to entropy $n$-tuples and entropy sets by
Huang \& Ye \cite{HY06}, {\red Dou, Ye \& Zhang \cite{DYZ06}} and Blanchard \& Huang \cite{BH}.

\begin{definition}\label{def:entropy_set}
Let $f:X \to X$ be a continuous map on a compact space $X$.
 A set $E \subset X$ is an {\em entropy set}
 if $\# E \geq 2$, and if for each open cover $\cU = \{U_1, \dots, U_n\}$ of $X$ with the property that for each $i$ there is $x \in E \setminus \overline{U_i}$, the topological entropy
 $\htop(f,\cU)$ (in the sense of Adler-Konheim-McAndrew \cite{AKM}) of the cover is positive.
 An entropy set of $n$ points is an {\em entropy $n$-tuple}, so an {\em entropy pair} is an entropy $2$-tuple.
\end{definition}

 Dou, Ye \& Zhang \cite{DYZ06} characterize maximal entropy sets, and prove that $E$ is an entropy set if and only if each $n$-tuple in $E$ with at least two distinct point is an entropy $n$-tuple.

The survey paper  of Glasner \& Ye \cite{GY09} is largely on the relation between the topological and ergodic approach to dynamical systems from the global point of view. In particular, they discuss topological and measure entropy pairs and $n$-tuples and present the relation between the topological entropy of open covers and the measure-theoretic entropy of finer partitions. The second part of the survey is on local properties of dynamical systems.
They introduce  $n$-tuples for an invariant measure, and establish their relation to topological entropy $n$-tuples. Glasner \& Ye \cite{GY09} discuss in detail the notion of entropy set,
in order to determine where the entropy is concentrated, {\red including the result of Dou, Ye and Zhang \cite[Corollary 3.3]{DYZ06}},
namely that any topological dynamical system with positive entropy admits an entropy set with infinitely many points. Moreover, for any topological dynamical system, there exists a compact countable subset such that its {\red $(n,\eps)$-separated set entropy} is equal to the entropy of the whole system, see \cite{YZ07}.

Glasner \& Ye discuss notions of completely positive entropy and  uniform positive entropy, introduced by Blanchard \cite{Bl92} in a topological context:
A topological dynamical system has {\em completely
positive entropy} if every non-trivial factor has positive topological entropy, and it has {\em uniform positive entropy}
if the topological entropy of every non-dense finite open cover is positive.
Uniform positive entropy implies weak mixing as well as completely positive entropy, which in turn implies the existence of a fully supported invariant measure.

\subsection{Neutralized entropy}
The paper by Garcia-Ramos \& Li \cite{GL24} is also a survey about recent developments in the local entropy theory for topological dynamical systems; they cover similar material as Glasner \& Ye, but for general (continuous) group actions.

In 1992, Thieullen \cite{T92} introduced a local form of entropy which he called $\alpha$-entropy and which is close to our approach of \tl entropy.
In his definition, the $\eps$ in the Bowen $(n,\eps)$-balls is replaced by an $n$-dependent quantity $e^{-\alpha n}$, but unlike our \tl entropy, he still uses Bowen balls.
Ben Ovadia \& Rodriguez-Hertz \cite{OR24} introduced {\em neutralized} (local) entropy as
 \begin{equation}\label{eq:neut}
\cE_\mu(x)  := \lim_{\alpha\to 0}\limsup_{n\to\infty} -\frac1n \log \mu(B_n(x, e^{-\alpha n})).
 \end{equation}
 This is the $\alpha$-entropy where eventually the limit $\alpha \to 0$ is taken.
Ben Ovadia \& Rodriguez-Hertz study the role it can play in hyperbolic dynamics and
compare it with the Brin-Katok local entropy, see~\eqref{eq:BK}.

They show that the neutralized local entropy for $C^{1+\beta}$ diffeomorphism  $f:M \to M$ of a compact closed manifold $M$, calculated with respect to an $f$-invariant probability measure $\mu$, coincides with Brin-Katok local entropy almost everywhere.

In a recent paper \cite{YCZ23}, Yang, Chen \& Zhou, extending the variational principle~\eqref{eq:FH} of Feng \& Huang \cite{FH12}, obtained a variational principle for neutralized entropy of a continuous map $f:X \to X$ on a compact metric space.
For a non-empty compact set $K \subset X$, they prove the equality
$$
\htop^{\widetilde{B}}(f,K)=\lim_{\alpha \to 0}\ \sup \{  \underline{h}_{\mu}^{\widetilde{BK}}(f,\alpha): \mu \text{ \red is a Borel probability measure with }  \mu(K)=1\},
$$
where $\htop^{\widetilde{B}}(f,K)$ denotes a neutralized version of the Bowen topological entropy of $f$ {\red recalled in Section~\ref{sec:toppressure}}. 

and $\underline{h}_{\mu}^{\widetilde{BK}}(f,\alpha)$ is the {\em lower neutralized Brin-Katok local entropy} of $\mu$, defined as
$$
\underline{h}_{\mu}^{\widetilde{BK}}(f,\alpha):=\int_X \liminf _{n \to \infty} \ -\frac1n \log \mu(B_n(x,e^{-n \alpha})) \ d\mu(x).
$$
{\red Although neutralized entropy $\cE_\mu(x)$ uses similar components as translocal entropy $h_\aa(x)$, they are hard to compare because
the former depends on the measure $\mu$ and the latter on the parameter $\aa$.}

{\bf Acknowledgments}
The first author was partially supported by the inner \L\'od\'z University Grant 11/IDUB/DOS/2021. This research was finalized during the visit of the second author to the Jagiellonian
University funded by the program Excellence Initiative – Research University at the Jagiellonian University in Krak\'ow.
{\red Both authors are very grateful for the careful reading and the many due corrections by the anonymous referees.}

\end{document}